\documentclass[12pt,a4paper]{article}

\usepackage[latin1]{inputenc}
\usepackage[english]{babel}
\usepackage[usenames, dvipsnames]{xcolor}
\usepackage{amssymb}
\usepackage{amsfonts}
\usepackage{amsmath}
\usepackage{amsthm}
\usepackage{epsfig}
\usepackage{graphics}
\usepackage{dsfont}

\usepackage{verbatim}

\usepackage{latexsym}

\newtheorem{theorem}{Theorem}[section]

\newtheorem{remark}[theorem]{Remark}

\newtheorem{proposition}[theorem]{Proposition}

\allowdisplaybreaks[4]

\newcommand{\bN}{\mathbb{N}}
\newcommand{\bR}{\mathbb{R}}
\newcommand{\bZ}{\mathbb{Z}}
\newcommand{\bC}{\mathbb{C}}
\newcommand{\law}{\mathcal{L}}
\newcommand{\bQ}{\mathbb{Q}}
\newcommand{\di}{\mathrm{d}}
\newcommand{\ri}{\mathrm{i}}
\newcommand{\re}{\mathrm{e}}

\newcommand{\one}{\mathbf{1}}

\begin{document}
\title{A Cram\'er--Wold device for infinite divisibility of  $\bZ^d$-valued distributions\thanks{This research was supported by DFG grant LI-1026/6-1. Financial support is gratefully acknowledged.}}
\author{David Berger\thanks{Technische Universit\"at Dresden,
Institut f\"ur Mathematische Stochastik, D-01062 Dresden, Germany, email: david.berger2@tu-dresden.de} \and Alexander Lindner\thanks{Ulm University, Institute of Mathematical Finance, D-89081 Ulm, Germany;
email:
alexander.lindner@uni-ulm.de}}
\date{\today}
\maketitle

\begin{abstract}
We show that a Cram\'er--Wold device holds for infinite divisibility of $\bZ^d$-valued distributions, i.e. that the distribution of a $\bZ^d$-valued random vector $X$ is infinitely divisible if and only if $\law(a^T X)$ is infinitely divisible for all $a\in \bR^d$, and that this in turn is equivalent to infinite divisibility of $\law(a^T X)$ for all $a\in \bN_0^d$. A key tool for proving this is a L\'evy--Khintchine type representation with a signed L\'evy measure for the characteristic function of a $\bZ^d$-valued distribution, provided the characteristic function is zero-free.
\end{abstract}


\section{Introduction}

In 1936, Cram\'er and Wold \cite{CramerWold1936} showed that a probability distribution on $\bR^d$ is uniquely determined by its one-dimensional projections, or in other words, the distribution of an $\bR^d$-valued random vector $X$ is determined by the distributions of all its linear combinations $a^T X$ when $a$ ranges in $\bR^d$. As a consequence, a sequence $(X_n)_{n\in \bN}$ of $\bR^d$-dimensional random vectors converges in distribution to an $\bR^d$-valued random vector $X$ if and only if $a^T X_n$ converges in distribution to $a^T X$ for all $a\in \bR^d$. These facts are the (classical)  Cram\'er--Wold device. It has become custom to call statements for multivariate random vectors that can be expressed by corresponding statements for the linear combinations also as Cram\'er--Wold devices.
For example, if $X$ is a random vector in $\bR^d$, then it is strictly stable if and only if $a^T X$ is strictly stable for all $a\in \bR^d$, it is symmetric stable if and only if $a^T X$ is symmetric stable for all $a\in \bR^d$, and it is stable with stability index greater than or equal to 1 if and only if the same is true for $a^T X$ for all $a\in \bR^d$, see e.g. \cite[Thm. 2.1.5]{SamorodnitskyTaqqu1994}. All these statements can be called a Cram\'er--Wold device for the corresponding property. It is however not true that a distribution on $\bR^d$ is stable if all its linear combinations $a^T X$ are stable, see e.g. \cite[Section 2.2]{SamorodnitskyTaqqu1994}.
Another example is given by multivariate regular variation at infinity of random vectors: Basrak et al. \cite{Basrak} showed that a Cram\'er--Wold device holds when the regular variation index is not an integer, and Hult and Lindskog \cite{Hult} gave a counter example when the regular variation index is an integer.

In this paper, we shall be interested in infinite divisibility of random vectors. If a random vector $X$ is infinitely divisible, then obviously $a^T X$ is infinitely divisible, too, for all $a\in \bR^d$, but the converse is not true. Counter examples include one by Dwass and Teicher \cite{DT57} who considered a three-dimensional Wishart distribution, and a two-dimensional distribution constructed by Ibragimov \cite{Ibragimov72}.
The example of Ibragimov is such that it is not infinitely divisible but quasi--infinitely divisible, meaning that the characteristic function of it has a L\'evy--Khintchine type representation with a \lq signed L\'evy measure\rq~rather than a L\'evy measure. In view of the existence of counter examples to the general problem, it is natural to ask if a Cram\'er--Wold device holds for subclasses of infinitely divisible distributions, as it does for symmetric stable random vectors or strictly stable random vectors as seen above. In this paper, we shall consider $\bZ^d$-valued distributions. Our main goal is to show that a Cram\'er--Wold device holds for such distributions, i.e. that (the distribution of) a $\bZ^d$-valued random vector $X$ is infinitely divisible if and only if $a^T X$ is infinitely divisible for all $a\in \bR^d$. Since $a^T X$ may no longer be  integer valued when $a\notin \bZ^d$, we also show that a $\bZ^d$-valued  random vector $X$ is infinitely divisible if and only if $a^T X$ is infinitely divisible for all $a\in \bN_0^d$, so that one does not leave the class of integer--valued distributions.

A crucial tool for proving the above mentioned  Cram\'er--Wold device in Section \ref{S4} will be to find a L\'evy--Khintchine type representation for the characteristic function of $\bZ^d$-valued probability distributions with a \lq signed L\'evy measure\rq~rather than a L\'evy measure. Such a representation will be found in Theorem \ref{discrete}. Distributions with L\'evy--Khintchine type representation of the characteristic function with a signed L\'evy measure are called \emph{quasi-infinitely divisible distributions}. Recently, they have attracted some attention and found new applications, in particular in one dimension, see e.g. \cite{Berger, ChhaibaDemniMouayn2016, DemniMouayn2015, Khartov2019, Lindner, Nakamura15, Passeggeri2020a, ZhangLiuLi}, to name just a few, but they have already appeared in the 1960s and 70s with works by Linnik \cite{Linnik64},  Cuppens \cite{Cuppens75} and others, some of these early works being also multivariate.   Theorem \ref{discrete} of the present paper can be seen as a multivariate extension of Theorem 8.1 in the recent paper \cite{Lindner}.

The paper is organised as follows. In the next section, we will set some notation and recall some definitions. In Section \ref{S3} we will state and prove the previously stated result about the representation of a zero--free characteristic function of a $\bZ^d$-valued distribution.  With that at hand in Section \ref{S4} we will be able  to get the desired Cram\'er--Wold device for infinite divisibility of $\bZ^d$--valued distributions.

\section{Notation and preliminaries}

We write $\bN = \{1,2,3,\ldots \}$, $\bN_0 = \bN \cup \{0\}$ and $\bZ, \bQ,\bR,\bC$ for the integers, the rational numbers, the real numbers and the complex numbers, respectively. The imaginary part of a complex number $z\in \bC$ will be denoted by $\Im (z)$, the complex conjugate of $z$ by $\overline{z}$.

The Euclidian norm in $\bR^d$ is denoted by $|\cdot|$, we write $a\wedge b = \min \{a,b\}$ for $a,b\in \bR$ and the indicator function of a set $A\subset \bR^d$ is denoted by $\one_A$. The canonical basis in $\bR^d$ will be denoted by $e_1,\ldots, e_d$, where $e_i$ is the $i$'th unit vector having all its entries zero apart from the $i$'th entry which is 1. The Dirac measure at a point $x\in \bR^d$ will be denoted by $\delta_x$. When we speak of a measure or signed measure on $\bR^d$, we will always mean it to be defined on the Borel-$\sigma$-algebra of $\bR^d$. The convolution of two probability measures $\mu_1$ and $\mu_2$ on $\bR^d$ will be denoted by $\mu_1\ast \mu_2$.

Throughout, vectors $a$ in $\bR^d$ will be column vectors, and by $a^T$ we denote the transpose of a vector or matrix.  The law (i.e. distribution) of a random vector $X$ will be denoted by $\law(X)$.
The characteristic function of a probability measure $\mu$ on $\bR^d$ (with the Borel-set as $\sigma$-algebra) is given by
$$\varphi_\mu(z) = \widehat{\mu}(z) = \int_{\bR^d} \re^{\ri \langle z, x\rangle} \, \mu(\di x), \quad z\in \bR^d,$$
where $\langle z, x \rangle= z^T x$ denotes the Euclidian inner product on $\bR^d$.
When $X$ is a random vector we also write $\varphi_X$ instead of $\varphi_{\law(X)}$.
By definition, a distribution $\mu$ on $\bR^d$ is infinitely divisible, if and only if for every $n\in \bN$ there exists some probability distribution $\mu_n$ on $\bR^d$ such that $\mu = \mu_n^{\ast n}$, the $n$-fold convolution of $\mu_n$ with itself. By the L\'evy--Khintchine formula, a distribution $\mu$ on $\bR^d$ is infinitely divisible if and only if its characteristic function admits a representation of the form
\begin{equation} \label{eq-LK}
\widehat{\mu}(z)  = \exp \left\{ \ri \gamma^T z - \frac12 z^T A z + \int_{\bR^d} \left( \re^{\ri z^T x} - 1 - \ri z^T x \one_{[-1,1]} (|x|) \right) \, \nu(\di x)\right\}
\end{equation}
with a symmetric positive semidefinite matrix $A\in \bR^{d\times d}$, a constant $\gamma\in \bR^d$ and a L\'evy measure $\nu$ on $\bR^d$, i.e. a measure satisfying $\nu(\{0\}) = 0$ and $\int_{\bR^d} (|x|^2 \wedge 1)\, \nu(\di x) < \infty$. The triplet $(A,\nu,\gamma)$ is then unique and called the characteristic triplet of $\mu$. For this and further information regarding infinitely divisible distributions we refer to Sato \cite{Sato}.

\section{A representation for the characteristic function of a $\bZ^d$-valued distribution} \label{S3}
\setcounter{equation}{0}

In this section we show that if the characteristic function of a $\bZ^d$-valued distribution has no zeroes, then it has  a L\'evy--Khintchine type representation with a finite signed \lq L\'evy measure\rq. This generalises a corresponding one-dimensional result in \cite[Thm. 8.1]{Lindner} to the multivariate setting. 

Let us first recall some definitions and facts. A function $f:\bR^d \to \bC$ is called \emph{$2\pi$-periodic in all coordinates}, if $f(z) = f(z + 2\pi e_i)$ for all $z\in \bR^d$ and $i\in \{1,\ldots, d\}$, where $e_1, \ldots, e_d$ denotes the canonical basis in $\bR^d$. A function which is $2\pi$-periodic in all coordinates can be identified with a function on the torus $\bR^d / (2\pi \bZ^d)$, and it will be continuous if the corresponding identification on the torus is continuous. The $d$-dimensional Wiener algebra $W^d$ consists of all continuous functions $f:\bR^d \to \bC$ which are $2\pi$-periodic in all coordinates and can be represented in the form $f(z) = \sum_{n\in \bZ^d} a_n \re^{\ri \langle n, z \rangle}$ with absolutely summable Fourier coefficients $(a_n)_{n\in \bZ^d}$. Since $W^d$ is an algebra, the product of two functions in $W^d$ is again in $W^d$. The Wiener-L\'evy theorem for Fourier series in several variables then states that whenever $f\in W^d$ and $h:D \to \bC$ is a holomorphic function in an open neighbourhood of $f(\bR^d)=f([0,2\pi]^d)$, then also $h\circ f \in W^d$; see e.g. Rudin  \cite[Thm. 6.2.4, p. 133]{Rudin}, where this is stated in the general context of Fourier analysis on locally compact abelian groups, and the version we need follows from the observation that the dual group of $G= \bZ^d$ is the torus $\Gamma=\bR^d / (2\pi \bZ^d)$.

Let us also recall that to  every continuous function $f:\bR^d \to \bC$ satisfying $f(0) = 1$ and $f(z) \neq 0$ for all $z\in \bR^d$ there exists a unique continuous function $\psi:\bR^d \to \bC$ satisfying $\psi(0) = 0$ and $\re^{\psi(z)} = f(z)$ for all $z\in \bR^d$; this unique function $\psi$ is called the \emph{distinguished logarithm} of $f$, see e.g.
\cite[Lem.~7.6]{Sato}.

In order to apply the Wiener--L\'evy theorem to the distinguished logarithm under certain conditions, we have to shift it into the right half of the complex plane, where we can use the principal branch of the logarithm. This is the contents of the next result, which is an extension to the multivariate setting of a corresponding result by Calderon et al.~\cite[Sect.~2, Lemma]{Calderon}.

\begin{proposition}\label{prop1}
Let $f:\bR^d\to \bC$ be an element in $W^d$ such that $f(z)\neq 0$ for all $z\in\bR^d$. Denote by $\psi$ the distinguished logarithm of  $f$ and assume that  $\psi$ is $2\pi$-periodic in all coordinates. Then $\psi\in W^d$.
\end{proposition}

\begin{proof}
Since $\psi$ is continuous and $2\pi$-periodic in every coordinate,   by the Weierstrass approximation theorem with trigonometric polynomials (e.g. \cite[Cor. 3.2.2., p.~183]{Grafakos}) there exists a finite subset $\Gamma\subset \bZ^d$ and a  trigonometric polynomial $p(z) = \sum_{k\in \Gamma} b_k \re^{\ri \langle k , z \rangle}$ with $b_k \in \bC$ such that $|\psi(z) - p(z)| < \pi/2$ for all $z\in \bR^d$. By the Wiener--L\'evy theorem applied to the exponential function we conclude that $\bR^d \ni z \mapsto \re^{- p(z)}$ is in $W^d$. Since $W^d$ is closed under multiplication, we conclude that also
\begin{align*}
\widetilde{f}: \bR^d \to \bC, \quad z \mapsto \widetilde{f}(z)=\re^{- p(z)} f(z) = \exp \left\{- p(z) + \psi(z)\right\}
\end{align*}
is in $W^d$. Since $\Im (\psi(z) -  p(z))\in (-\pi/2,\pi/2)$ we see that $\widetilde{f}(z)$ is in the right-half plane for every $z\in \bR^d$. It follows that $\bR^d \ni z\mapsto \psi(z) -  p(z)$ coincides with the principal branch of the complex logarithm applied to $\widetilde{f}$, so that  the Wiener--L\'evy theorem applied with this branch shows that $\bR^d \ni z \mapsto \psi(z) -  p(z)$ is in $W^d$. Since $\bR^d \ni z \mapsto  p(z)$ is in $W^d$ as a trigonometric polynomial, so must be $\psi$.
\end{proof}

We can now state and prove the aforementioned representation of the characteristic function of a $\bZ^d$-valued distribution, provided it is zero-free. Recall that a distribution $\mu$ on $\bR^d$ is a compound Poisson distribution with jump size distribution $\sigma$, where $\sigma$ is a probability measure on $\bR^d \setminus \{0\}$, if the characteristic function of $\mu$ has the form $\bR^d \ni z \mapsto \widehat{\mu}(z) = \exp \{ \int_{\bR^d} (\re^{\ri \langle z, x \rangle} - 1 ) \, \lambda \sigma(\di x)\}$ with some $\lambda>0$, cf. \cite[Def. 4.1]{Sato}. The constant $\lambda$ will be called the jump rate of $\mu$ (more precisely, one should say the jump rate of a compound Poisson process associated with $\mu$).

\begin{theorem} \label{discrete}
Let $\mu$ be a probability distribution on $\bZ^d$ with characteristic function $\widehat{\mu}$. Then the following are equivalent:
\begin{enumerate}
\item[(i)] $\widehat{\mu}(z) \neq 0$ for all $z\in \bR^d$.
\item[(ii)] There exist $k\in \bZ^d$ and a finite signed measure $\nu$ on $\bR^d$ with support contained in $\bZ^d \setminus \{0\}$  such that
\begin{equation} \label{eq-rep}
\widehat{\mu}(z) = \exp \left\{ \ri \langle k, z \rangle + \int_{\bR^d} \left( \re^{\ri \langle z, x \rangle} - 1 \right) \, \nu(\di x) \right\} \quad \forall\; z\in \bR^d.
\end{equation}
 \item[(iii)] There exist $k\in \bZ^d$ and  two compound Poisson distributions $\mu_1$ and $\mu_2$ with jump size distribution supported in $\bZ^d \setminus \{0\}$ such that
$$\widehat{\mu}(z) = \re^{\ri \langle k , z \rangle} \frac{\widehat{\mu}_1(z)}{\widehat{\mu}_2(z)} \quad \forall\; z\in \bR^d.$$
 \item[(iv)] There exist two infinitely divisible distributions $\mu_3$ and $\mu_4$ such that $\widehat{\mu}(z) = \widehat{\mu}_3(z) / \widehat{\mu}_4(z)$ for all $z\in \bR^d$.
 \end{enumerate}
\end{theorem}

\begin{proof}
To see that (ii) implies (iii), denote by $\nu^+$ and $\nu^-$ the positive and negative part of $\nu$ in its Hahn--Jordan decomposition. Define $\nu_1 := \nu^+ + \delta_{e_1}$, $\nu_2 := \nu^- + \delta_{e_1}$, $\lambda_i := \nu_i(\bR^d) > 0$ and $\sigma_i := \nu_i / \lambda_i$ for $i=1,2$ (the reason why we add $\delta_{e_1}$ to $\nu^+$ and $\nu^-$ is that this way we guarantee that $\lambda_1,\lambda_2 > 0$). Then the $\sigma_i$ are probability measures supported in $\bZ^d \setminus \{0\}$. Denote by $\mu_i$ the compound Poisson distribution with jump rate $\lambda_i$ and jump size distribution $\sigma_i$. Then $\widehat{\mu_i}(z) = \exp \left\{ \int_{\bR^d} \left( \re^{\ri \langle z, x \rangle} - 1 \right) \, \lambda_i \sigma_i(\di x)\right\}$ and (iii) follows readily. That (iii) implies (iv) is clear by choosing $\mu_3 = \delta_k \ast \mu_1$ and $\mu_4 = \mu_3$, as is the implication that (iv) implies (i) since the characteristic function of an infinitely divisible distribution has no zeroes, cf. \cite[Lem. 7.5]{Sato}.

It remains to show that (i) implies (ii). Let $\mu= \sum_{n\in \bZ^d} a_n \delta_n$ with $0\leq a_n \leq 1$ and $\sum_{n\in \bZ^d} a_n = 1$. Then $\widehat{\mu}(z) = \sum_{n\in \bZ^d} a_n \re^{\ri \langle n, z\rangle }$ for $z\in \bR^d$ so that $\widehat{\mu} \in W^d$. Denote the distinguished logarithm of $\widehat{\mu}$ be $\psi$; it exists by assumption (i).

Since $\widehat{\mu}$ is $2\pi$-periodic in every coordinate, we obtain for every $j\in \{1,\ldots, d\}$
$$\exp \left( \psi (z + 2\pi e_j) \right)  =  \widehat{\mu} (z  + 2 \pi e_j)
 =  \widehat{\mu} (z)\\
=  \exp \left( \psi (z) \right).$$
It follows that $\psi(z+2\pi e_j) - \psi(z) \in 2 \pi \ri \bZ$ for all $z\in \bR^d$ and $j\in \{1,\ldots, d\}$. Since for each fixed $j$, the function $\bR^d \ni z \mapsto  \psi(z+2\pi e_j) - \psi(z)$ is continuous, we conclude that
\begin{align*}
k_j:=\frac{\psi(z + 2\pi e_j)-
\psi(z)}{2\pi \ri} \in \bZ
\end{align*}
does not depend on $z\in \bR^d$. Denote $k:= (k_1,\ldots, k_d)^T \in \bZ^d$ and define the function $\widetilde{\psi}:\bR^d \to \bC$ by $\widetilde{\psi}(z) = \psi(z) - \ri \langle k, z \rangle$. Then $\widetilde{\psi}$ is continuous with $\widetilde{\psi}(0) = 0$ and $\exp (\widetilde{\psi}(z)) = \widehat{\mu}(z) \exp (- \ri \langle k, z \rangle) = (\mu \ast \delta_{-k})\,\,\widehat{ { } }\; (z)$. Hence $\widetilde{\psi}$ is the distinguished logarithm of  $\bR^d \ni z \mapsto \sum_{n\in \bZ^d} a_{n+k} \re^{\ri \langle n, z \rangle}\in W^d$. Since
\begin{eqnarray*}
\widetilde{\psi}(z+2\pi e_j) - \widetilde{\psi}(z) & = & \psi(z+2\pi e_j) - \psi(z) - \ri \langle k , 2\pi e_j\rangle = 2\pi \ri k_j - 2\pi \ri k_j =  0
\end{eqnarray*}
by the definition of $k_j$ and $k$ we conclude that $\widetilde{\psi}$ is $2\pi$-periodic in all coordinates. An application of Proposition \ref{prop1} now shows that $\widetilde{\psi} \in W^d$.
This implies that there are $c_n \in \bC$, $n\in \bZ^d$, with $\sum_{n\in \bZ} |c_n| < \infty$ such that $\widetilde{\psi}(z) = \sum_{n\in \bZ^d} c_n \re^{\ri \langle n , z \rangle}$ for all $z\in \bR^d$. Since $\sum_{n\in \bZ^d} c_n = \widetilde{\psi}(0) = 0$ we have $c_0 = -\sum_{n\in \bZ^d\setminus \{0\}} c_n$. Define the (finite) complex measure $\nu$ by
$$\nu := \sum_{n\in \bZ^d \setminus \{0\}} c_n \delta_n.$$
Then $\nu$ is supported in $\bZ^d \setminus \{0\}$ and $\nu(\bR^d) = -c_0$, hence
$$\widetilde{\psi}(z) = c_0 + \sum_{n\in \bZ^d \setminus \{0\}} \int_{\bR^d} \re^{\ri \langle z , x \rangle} c_n \, \delta_n (\di x) = c_0 + \int_{\bR^d} \re^{\ri \langle z, x \rangle} \nu(\di x) =
\int_{\bR^d} \left( \re^{\ri \langle z , x \rangle} - 1 \right) \, \nu(\di x)$$
for all $z\in \bR^d$. This is the desired representation \eqref{eq-rep}, apart from the fact that we still have to show that $\nu$ is a signed measure, i.e. that the $c_n$ are real-valued. To see this, let $X$ be a random variable with distribution ${\mu}$ and $a=(a_1,\ldots, a_d)^T \in \bR^d$.  Denote by $\nu_a'$ the image measure of $\nu$ under the mapping $\bR^d \to \bR$, $x\mapsto \langle a, x\rangle = a^T x$ and define the measure $\nu_a$ on $\bR$ by $\nu_a(\di y) = \mathbf{1}_{\bR \setminus \{0\}}(y) \, \nu_a'(\di y)$, i.e. a potential mass of $\nu_a'$ at 0 is eliminated. Then the characteristic function $\bR \ni u \mapsto \varphi_{\langle a, X \rangle}(u)$ of $a^T X$ is given by
\begin{eqnarray}
\varphi_{\langle a, X\rangle}(u) & = & \varphi_X(ua) = \exp \left\{ \ri  k^T a u+ \int_{\bR^d} \left( \re^{\ri \langle ua, x\rangle} -1 \right) \, \nu(\di x)\right\} \nonumber \\
& = & \exp \left\{ \ri  k^T a u + \int_{\bR^d} \left( \re^{\ri u y } - 1 \right) \, \nu_a'(\di y) \right\} \nonumber\\
& = &  \exp \left\{ \ri  k^T a u + \int_{\bR^d} \left( \re^{\ri u y } - 1 \right) \, \nu_a(\di y) \right\} . \label{eq-CW1}
\end{eqnarray}
This shows that the characteristic function of the 1-dimensional distribution $\law(a^T X)$ has a L\'evy--Khintchine type representation with a finite complex measure $\nu_a$. But as shown in \cite[Thm. 3.2]{Berger}, this is only possible if $\nu_a$ is a signed measure (i.e. the imaginary part of the complex measure $\nu_a$ is zero). We conclude that $\nu_a$ is real valued for all $a\in \bR^d$. Now choose $a=(a_1,\ldots, a_d)^T \in \bR^d$ such that $a_1,\ldots, a_d$ are linearly independent over $\bQ$. Then $a^T m \neq a^T n\neq 0$ for all $n,m\in \bZ^d\setminus \{0\}$ with $n\neq m$ and we conclude for every $n\in \bZ^d \setminus \{0\}$ that
\begin{equation} \label{eq-CW2}
\nu(\{n\}) = \nu ( \{ m \in \bZ^d \setminus \{0\} : a^T m = a^T n\}) = \nu_a ( \{a^T n\}).
\end{equation}
Since $\nu_a$ is a signed measure as seen before, the right-hand side of \eqref{eq-CW2} is in $\bR$, hence so is the left-hand side. This shows that $\nu$ is a signed measure, finishing the proof.
\end{proof}

\begin{remark} \label{rem-qid}
{\rm In \cite[Rem. 2.4]{Lindner}, a probability distribution $\mu$ on $\bR^d$ is called \emph{quasi-infinitely divisible} if its characteristic function can be written as the quotient of the characteristic functions of two infinitely divisible distributions, and it is remarked that this is equivalent to the fact that there exists a L\'evy--Khintchine type representation of the characteristic function with a \lq signed L\'evy measure\rq,  termed \emph{quasi-L\'evy measure}. With these definitions, Theorem \ref{discrete} implies that a probability distribution $\mu$ on $\bZ^d$ is quasi-infinitely divisible if and only if its characteristic function has no zeroes on $\bR^d$, and if this is the case, the quasi-L\'evy measure $\nu$ of $\mu$ is finite and supported in $\bZ^d\setminus \{0\}$ and the \lq drift\rq~$k$ of $\mu$ is in $\bZ^d$. In dimension 1, such a result has been obtained in \cite[Thm. 8.1]{Lindner}, so that Theorem \ref{discrete} can be seen as a multivariate extension of this result. The proof of the implication \lq (i) $\Longrightarrow$ (ii)\rq~has some similarities with that of \cite[Thm. 8.1]{Lindner}
 in the sense that the Wiener--L\'evy theorem is used, but our argument for showing that the measure $\nu$ is indeed signed (i.e. real valued) is quite different from that in
\cite[Thm. 8.1]{Lindner} and the argument given there does not seem to generalise easily to higher dimensions.}
\end{remark}

\section{A Cram\'er--Wold device for infinite divisibility of $\bZ^d$-valued distributions} \label{S4}

We can now show the aforementioned Cram\'er--Wold device for infinite divisibility of $\bZ^d$-valued distributions. We give several characterisations, which show that one does not have to \lq test\rq~ in all directions, but only in some \lq nice\rq~directions, or if one knows a priori that the characteristic function is zero-free, then one has to test only in one specific direction. The precise statement is as follows:

\begin{theorem} \label{thm-CW}
Let $X$ be a $\bZ^d$-valued valued random vector with distribution $\mu$. Then the following are equivalent:
\begin{enumerate}
\item[(i)] $\mu$ is infinitely divisible.
\item[(ii)] $\law(a^T X)$ is infinitely divisible for all $a\in \bR^d$.
\item[(iii)] $\law(a^T X)$ is infinitely divisible for all $a\in \bN_0^d$.
\item[(iv)] The characteristic function $\widehat{\mu}$ of $\mu$ has no zeroes  on $\bR^d$ and there exists some $a=(a_1,\ldots, a_d)^T \in \bR^d$ such that $a_1,\ldots, a_d$ are linearly independent over $\bQ$ and such that $\law(a^T X)$ is infinitely divisible.
\item[(v)] There exists a sequence $(a(n))_{n\in \bN}$ of vectors in $\bR^d \setminus \{0\}$ such that the closure of $\{ |a(n)|^{-1} a(n) :n\in \bN\}$ has non-empty interior in the relative topology of the unit sphere $S^{d-1} = \{ x\in \bR^d: |x|=1\}$ and such that $\law(a(n)^T X)$ is infinitely divisible for all $n\in \bN$.
    \end{enumerate}
\end{theorem}

\begin{proof}
That (i) implies (ii) is well-known, see e.g. Sato \cite[Prop. 11.10]{Sato}. That (ii) implies (iii) is clear. That (iii) implies (v) is easily seen from the fact that $\{ |a|^{-1} a : a \in \bN_0^d \setminus \{0\}\}$ is dense in $S^{d-1} \cap [0,\infty)^d$.

Let us show that (v) implies (iv). Write $b(n) := |a(n)|^{-1} a(n) \in S^{d-1}$, denote the closure of $\{b(n): n \in \bN\}$ by $C$ and the interior of $C$ in the relative topology of $S^{d-1}$ by $U$. Since $\law(a(n)^T X)$ is infinitely divisible for all $n\in \bN$, so is $\law(b(n)^T X)$ for each $n\in \bN$. Since $\law(c(n)^T X)$ converges weakly to $\law(c^T X)$ if the sequence $c(n)\in \bR^d$ converges to $c\in \bR^d$, and since the class of all infinitely divisible distributions is closed under weak convergence (e.g. \cite[Lem. 7.8]{Sato}), it follows that also $\law(c^T X)$ is infinitely divisible for all $c\in C$. Since the interior $U$ of $C$ in the relative topology of $S^{d-1}$ is non-empty, there exists some $a=(a_1,\ldots, a_d)^T \in U \subset S^{d-1}$ such that $a_1,\ldots, a_d$ are linearly independent over $\bQ$. Since $a\in U\subset C$ we  see that $\law(a^T X)$ is infinitely divisible. It remains to show that $\widehat{\mu}$ is zero-free. To see this, observe that with $\law(c^T X)$ being infinitely divisible for $c\in U$, also $\law(r c^T X)$ is infinitely divisible for all $r>0$. But since $U$ is non-empty and open in the relative topology of $S^{d-1}$, for large enough $R>0$ the set $U_R := \{r c : r \in (R,\infty) \subset \bR, b \in U\} \subset \bR^d$ will fully contain a cube of the form $Q := v + [0,2\pi]^d$ with some $v \in \bR^d$. Since $\law(z^T X)$ is infinitely divisible for all $z\in U_R$, its characteristic function is zero-free (cf. \cite[Lem. 7.5]{Sato}), in particular $\widehat{\mu}(z) = \varphi_X (z) = \varphi_{\langle z, X\rangle}(1) \neq 0$ for all $z\in U_R$. Since $U_R$ contains $Q$ for large enough $R$, we have $\widehat{\mu}(z) \neq 0$ for all $z\in Q$ and hence $\widehat{\mu}(z) \neq 0$ for all $z\in \bR^d$ since $\widehat{\mu}$ is $2\pi$-periodic in all coordinates. This is the condition (iv).

It remains to show that (iv) implies (i). Since $\widehat{\mu}(z) \neq 0$ for all $z\in \bR^d$, $\widehat{\mu}$ has the representation \eqref{eq-rep} with $k\in \bZ^d$ and the signed finite measure $\nu$ supported in $\bZ^d \setminus \{0\}$. Let $a=(a_1,\ldots, a_d)^T$ as in (iv). As in the proof of Theorem \ref{discrete}, denote by $\nu_a'$ the image measure of $\nu$ under the mapping $\bR^d \to \bR$, $x\mapsto a^T x$ and define $\nu_a(\di y) :=   \mathbf{1}_{\bR \setminus \{0\}} (y) \, \nu_a'(\di y)$ (actually, since the components of $a$ are linearly independent over $\bQ$ and $\nu$ is supported in $\bZ^d \setminus \{0\}$, it is easily seen that $\nu_a = \nu_a'$ here).  From Equation \eqref{eq-CW1} we see that $\law(a^T X)$ has  a L\'evy--Khintchine type representation with the finite \lq signed L\'evy measure\rq~$\nu_a$. But whenever a probability has such a representation with a signed L\'evy measure, then it is known that this signed L\'evy measure is unique, see e.g. Sato \cite[Exercise 12.2]{Sato}. But since $\law(a^T X)$ is infinitely divisible by assumption, we see that $\nu_a$ must actually be a L\'evy measure, so a non-negative measure and it follows that $\nu_a(B) \geq 0$ for all Borel subsets $B$ of $\bR$. In particular, $\nu_a (\{a^T n\}) \geq 0$ for all $n\in \bZ^d \setminus \{0\}$. Using \eqref{eq-CW2} it then follows that $\nu(\{n\}) \geq 0$ for all $n\in \bZ^d \setminus \{0\}$, showing that $\nu$ is a L\'evy measure. The representation \eqref{eq-rep} is then a true L\'evy--Khintchine representation of the characteristic function of $\mu$, showing that $\mu$ is infinitely divisible.
\end{proof}

\begin{remark}
{\rm
(a) Let  $\mu=\law(X)$ be an $\bN_0^d$-valued distribution. By Theorem  \ref{thm-CW}, $\mu$ is infinitely divisible if and only if $\law(a^T X)$ is infinitely divisible for all $a\in \bN_0^d$. Since  $a^T X$ is $\bN_0$-valued, known characterisations or sufficient conditions for infinite divisibility of  $\bN_0$-valued random variables can be applied such as Katti's criterion (cf. \cite[Cor. 51.2]{Sato}) or log-convexity of the counting density (cf. \cite[Thm. 51.3]{Sato}) to check if $\law(a^T X)$ and hence $\law(X) = \mu$ are infinitely divisible.\\
(b) Let $Y$ be a $\bZ^d$-valued random vector, $A$ an invertible $\bR^{d\times d}$ matrix and $v\in \bR^d$. Denote $X:= A Y + v$ Then it is well known that $\law(X)$ is infinitely divisible if and only if $\law(Y)$ is. Since $a^T Y = (a^T A) X + a^T v$, it is easily seen that the characterisations of Theorem \ref{thm-CW} carry over word by word to $X=AY+v$, i.e. to distributions concentrated on the lattice $v + A \bZ^d$.
}
\end{remark}

\end{document}